\documentclass[12pt]{amsart}
\usepackage{latexsym, amssymb, amsmath}

\def\vol{\mathrm{Vol}}

\newtheorem{theorem}{Theorem}[section]

\newtheorem{lemma}[theorem]{Lemma}
\newtheorem{corollary}[theorem]{Corollary}
\newtheorem{proposition}[theorem]{Proposition}

\theoremstyle{definition}
\newtheorem{definition}[theorem]{Definition}

\newtheorem*{example}{Example}

\newtheorem{claim}{Claim}

\makeatletter
    
    \@addtoreset{equation}{section}
  \makeatother

\newcommand{\Ric}{{\rm Ric}}
\newcommand{\grad}{{\rm grad}\,}

\newcommand{\n}{\nabla}

\newcommand{\dive}{{\rm div}\,}
\newcommand{\tr}{{\rm Trace}\,}
\newcommand{\Rm}{{\rm Rm}\,}

\begin{document}

\title[Self-similar solutions to the Hesse flow]
{Self-similar solutions to the Hesse flow}



\author{Shun Maeta}
\address{Department of Mathematics,
 Shimane University, Nishikawatsu 1060 Matsue, 690-8504, Japan.}
\curraddr{}
\email{shun.maeta@gmail.com~{\em or}~maeta@riko.shimane-u.ac.jp}
\thanks{The author is partially supported by the Grant-in-Aid for Young Scientists, No.19K14534, Japan Society for the Promotion of Science.}

\date{}

\dedicatory{}


\commby{}

\begin{abstract}
We define a Hesse soliton, that is, a self-similar solution to the Hesse flow on Hessian manifolds.
On information geometry, the $e$-connection and the $m$-connection are important, which do not coincide with the Levi-Civita one. 
Therefore, it is interesting to consider a Hessian manifold with a flat connection which does not coincide with the Levi-Civita one. We call it a proper Hessian manifold. 
In this paper, we show that any compact proper Hesse soliton is expanding and any non-trivial compact gradient Hesse soliton is proper. Furthermore, we show that the dual space of a Hesse-Einstein manifold can be understood as a Hesse soliton.
\end{abstract}

\maketitle


\bibliographystyle{amsplain}

\section{Introduction}\label{intro}

Machine learning might be one of the most powerful tool for the human race.
Information geometry is regarded as a basic and important mathematical research in this field, which was introduced by S. Amari.
In particular, Exponential families and Mixture families of probability distributions are important in information geometry. They have the dual flat structure. Hence, interestingly, they have Hessian structure.

Information geometry is built on the basis of differential geometry.  
The geometric flow is one of the most powerful tool in the theory of differential geometry. In particular, the Ricci flow is powerful.
In fact, as is well known, Poincare conjecture was solved by the Ricci flow.
The Ricci flow is defined as follows by using the Ricci tensor $\Ric(g(t))$ (cf. \cite{Hamilton82}): 
$$\frac{\partial}{\partial t}g(t)=-2\Ric(g(t)),$$
where $g(t)$ is the time dependent Riemannian metric on a Riemannian manifold $(M,g(t))$.
Therefore, Hessian manifolds can be deeply understood by considering the geometric flow for the second Koszul form $\beta$. In fact, $\beta$ plays a similar role to that of the Ricci tensor in a Hessian manifold (i.e., a manifold with a Hessian structure).
The flow is called the Hesse flow (or the Hesse-Koszul flow) defined by M. Mirghafouri and F. Malek (cf. \cite{MM17}, see also \cite{PT20}):
$$\frac{\partial}{\partial t}g(t)=2\beta(g(t)).$$
They proved the short-time existence, the global existence and the uniqueness of it. 
S. Puechmorel and T. D. T\^o \cite{PT20} studied some convergence theorems of it on compact Hessian manifolds.

A self-similar solution i.e., a soliton equation plays an important and fundamental role in the study of a geometric flow.
In fact, the Ricci soliton which is the self-similar solution to the Ricci flow plays an important role in solving Poincare conjecture and the geometrization conjecture. 

Therefore, in this paper, we define the self-similar solution to the Hesse flow and study it.
In particular, we show that the expanding case is important in compact Hesse solitons. We also show that one can understand that a non trivial gradient Hesse soliton is interesting from the point of view information geometry. In particular, under the second Koszul form coincides with the dual one, any compact gradient Hesse soliton must be Hesse-Einstein. Hesse-Einstein manifolds can be regarded to a notion of Einstein manifolds in Riemannian geometry.
Furthermore, we show that one can understand the dual space of a Hesse-Einstein manifold as a Hesse soliton.


\quad\\

\section{Preliminary}
In this section, we set up terminology and define some notions which are related to Hessian manifolds and information geometry. 

\quad\\
\subsection{Riemannian geometry}~

Let $(M,g)$ be an $n$-dimensional Riemannian manifold. As is well known, the Levi-Civita connection $\nabla:TM\times C^\infty(TM)\rightarrow C^\infty(TM)$ is the unique connection on $TM$, which is compatible with the metric and is torsion free:
$$Xg(Y,Z)=g(\nabla_XY,Z)+g(Y,\nabla_XZ),$$
$$\nabla_XY-\nabla_YX=[X,Y].$$
The Riemannian curvature tensor is defined by 
$$Rm(X,Y)Z=\n_X\n_YZ-\n_Y\n_XZ-\n_{[X,Y]}Z,$$
for any vector field $X,Y,Z\in\mathfrak{X}(M).$
We use the notations ${R^i}_{jkl}$ as
$$Rm\left(\frac{\partial}{\partial x_k},\frac{\partial}{\partial x_l}\right)\frac{\partial}{\partial x_j}=\sum_i^n{R^i}_{jkl}\frac{\partial}{\partial x_i},$$
and 
$R_{ijkl}=g^{ip}{R^p}_{jkl}$.
The Ricci and scalar curvatures are defined by $R_{ij}=R_{ijik}$ and $R=R_{ii}$.

\quad\\

\subsection{Hessian manifolds}~

Let $M$ be an $n$-dimensional smooth manifold. A connection $D$ is said to be flat if $D$ satisfies that it is torsion free and the curvature tensor $Rm^D$ vanishes, that is,
$$D_XY-D_YX=[X,Y],$$
and 
$$Rm^D(X,Y)Z:=D_XD_YZ-D_YD_XZ-D_{[X,Y]}Z=0.$$

\begin{definition}
A Riemannian metric $g$ on a flat manifold $(M,D)$ is called a Hessian metric if $g$ can be locally expressed by
$$g=Dd\varphi,$$
for some smooth function $\varphi$, that is, 
$$g_{ij}=\frac{\partial^2 \varphi}{\partial x_i\partial x_j},$$
for an affine coordinate system. $(D,g)$ is called a Hessian structure. $(M,D,g)$ is called a Hessian manifold.
\end{definition}

Let $D'=2\nabla-D$, then $D'$ is also a flat connection and $(D',g)$ is a Hessian structure. $D'$ is called the dual connection and $(D',g)$ is called the dual Hessian structure of $(D,g)$.

The flat connection $D$ and the dual one $D'$ satisfy that
$$Xg(Y,Z)=g(D_XY,Z)+g(Y,D'_XZ).$$

\begin{definition}
Let $v_g$ be the volume form of $g$ and $X$ be a vector field on $M$. The first and second Koszul forms $\alpha$ and $\beta$ of $(M,D)$ are defined by 
\begin{equation}\label{1st}
D_Xv_g=\alpha(X)v_g,
\end{equation}
\begin{equation}\label{2nd}
\beta=D\alpha.
\end{equation}
We denote by $\gamma$ the difference tensor of $\nabla$ and $D$:
$$\gamma_XY=\nabla_XY-D_XY.$$
Here we remark that since $D_{\partial_i}\partial_j=0$, the components ${\gamma^i}_{jk}$ of $\gamma$ with respect to affine coordinate systems coincide with the Christoffel symbols ${\Gamma^i}_{jk}$ of $\nabla$, where $\partial_i=\frac{\partial}{\partial x_i}$.

A tensor field $H$ of type $(1,3)$ defined by the covariant differential 
$$H=D\gamma$$
of $\gamma$ is said to be the Hessian curvature tensor for $(D,g)$.
The components ${H^i}_{jkl}$ of $H$ is given by
$${H^i}_{jkl}=\frac{\partial {\gamma^i}_{jl}}{\partial x_k}.$$
\end{definition}

\begin{proposition}[Proposition 3.4 in \cite{Shima07}]\label{ab} On a Hessian manifold, the following holds:

$(1)$ $\alpha(X)=\tr \gamma_X$.

$(2)$ $\alpha_i={\gamma^r}_{ri}$.

$(3)$ $\beta_{ij}={H^r}_{rij}={H_{ijr}}^r$.
\end{proposition}

The second Koszul form $\beta$ plays a similar role to that of the Ricci tensor and one can define Hesse-Einstein manifolds:

\begin{definition}
Let $(M,D,g)$ be a Hessian manifold. If the following holds
$$\beta=\lambda g,$$
then $(M,D,g)$ is called a Hesse-Einstein manifold.
\end{definition}

\begin{proposition}[Proposition 2.3 in \cite{Shima07}]\label{curv}
The curvature tensor of a Hessian manifold $(M,D)$ is given as follows.

\begin{equation}\label{rm}
{R^{i}}_{jkl}={\gamma^i}_{lr} {\gamma^r}_{jk} -{\gamma^i}_{kr} {\gamma^r
}_{jl}.
\end{equation}
\end{proposition}
This implies the following.
\begin{align}\label{ric}
R_{jk}={R^s}_{jsk}
={\gamma^s}_{kr} {\gamma^r}_{js} -\alpha_r {\gamma^r
}_{jk},
\end{align}
and
\begin{align}\label{scal}
R =|{\gamma}|^2 -|\alpha|^2.
\end{align}

We can use the following notations without confusion: For example, 
$\gamma_{ijk}\gamma_{ist}={\gamma^i}_{jk}\gamma_{ist},$ $H_{rrij}={H^r}_{rij}$, $\alpha_r\gamma_{rij}=\alpha^r\gamma_{rij}$, etc.
\quad\\

\subsection{Information geometry}~

Hessian manifolds play an important role in information geometry:
 For $\Omega_n=\{1,2,\cdots,n\}$, let
$$S_{n-1}:=\left\{p:\Omega_n\rightarrow \mathbb{R}_+;\sum_{\omega\in\Omega_n} p(\omega)=1\right\}$$ be a set of all probability distribution on $\Omega_n,$
where $\mathbb{R}_+=\{x\in\mathbb{R};x>0\}$.
As is well known, one can regard it as a manifold (see for example \cite{Fujiwara}).
A metric $g^F$ on $S_{n-1}$ such as
$g^F_p(X,Y)=\sum_{\omega=1}^n p(\omega)(X \log p(\omega))(Y\log p(\omega))$
is called a Fisher information metric.
For each $\alpha\in\mathbb{R}$, $\nabla^{(\alpha)}$ is determined by
$$g^F_p(\nabla^{(\alpha)}_XY,Z)=g^F_p(\nabla_XY,Z)-\frac{\alpha}{2}\sum_{\omega=1}^np(\omega)(X \log p(\omega))(Y\log p(\omega))(Z\log p(\omega)),$$
where $\nabla$ is the Levi-Civita connection compatible with $g^F.$
$\nabla^{(\alpha)}$ is called the $\alpha$-connection.

Chentsov (cf. \cite{Chentsov}) shows that an extremely natural invariance requirement of $S_{n-1}$ determines a metric and a connection of $S_{n-1}$, that is, the metric is the Fisher information metric and the connection is the $\alpha$-connection on $S_{n-1}$. This means that on information geometry, the $\alpha$-connection is the most natural connection.

The $\alpha$-connection $\nabla^{(\alpha)}$ satisfies that
$$Xg^F(Y,Z)=g^F(\nabla^{(\alpha)}_XY,Z)+g^F(Y,\nabla^{(-\alpha)}_XZ).$$
The most important case is $\alpha=1$.
It is known that for $(g^F,\nabla^{(1)},\nabla^{(-1)})$, $S_{n-1}$ is the dual flat manifold and $g^F$ can be written $g^F_{ij}=\partial_i\partial_j\varphi$ for an affine coordinate system (cf \cite{Amari}).
Hence, $(S_{n-1},\nabla^{(1)},g^F)$ is a Hessian manifold with $\nabla^{(1)}\not=\nabla.$

Therefore, to apply Hessian geometry to information geometry, it is important to consider a Hessian manifold $(M,D,g)$ with $D\not=\nabla.$ From this, we define the following:
\begin{definition}
Let $(M,D)$ be a Hessian manifold. If $D\not=\nabla$, $M$ is called a proper Hessian manifold.
\end{definition}

\quad\\


\section{Hesse solitons}
Let $(M,D,g)$ be a Hessian manifold. In this section, we consider self-similar solutions to the Hesse flow $\partial_t g=2\beta$.
We first consider the Hesse flow from the point of view of the Laplacian on Hessian manifolds.
H. Shima \cite{Shima86} considered the Laplacian on Hessian manifolds.
\begin{definition}[\cite{Shima86}]
Let $\mathcal{A}^{p,q}$ be the tensor product $(\overset{p}{\wedge}TM)\otimes(\overset{q}{\wedge}T^*M)$.
Let $v_g$ be the volume element determined by $g$. We identify $v_g$ with $v_g\otimes1\in\mathcal{A}^{n,0}$ and set $\overline v_g=1\otimes v_g\in \mathcal{A}^{0,n}$.
For any vector field $X$, we define interior product operators by
$$i(X):\mathcal{A}^{p,q}\rightarrow \mathcal{A}^{p-1,q},~~i(X)\omega=\omega(X,\cdots;\cdots),$$
$$\overline i(X):\mathcal{A}^{p,q}\rightarrow \mathcal{A}^{p,q-1},~~\overline i(X)\omega=\omega(\cdots;X,\cdots).$$

A coboundary operator $\partial: \mathcal{A}^{p,q}\rightarrow\mathcal{A}^{p+1,q}$ is  defined by
$$\partial=e(dx_i)D_{\frac{\partial}{\partial x_i}},$$
where $e$ is an exterior product operator defined by
$$e(\omega):\eta\in\mathcal{A}^{p,q}\rightarrow \omega\wedge \eta\in\mathcal{A}^{p+r,q+s},$$ for $\omega\in\mathcal{A}^{r,s}.$

The adjoint operator of $\partial$ is denoted by $\delta=(-1)^p\star^{-1}\partial\star$ on the space $\mathcal{A}^{p,q}$
,
where $\star$ is the star operator $\star:\mathcal{A}^{p,q}\rightarrow \mathcal{A}^{n-p,n-q}$ defined by
\begin{align*}
(\star \omega)&(X_1,\cdots,X_{n-p};Y_1,\cdots,Y_{n-q})v_g\wedge \overline v_g\\
=&~\omega\wedge\overline i(X_1)g\wedge\cdots \wedge \overline i (X_{n-p})g\wedge i(Y_1)g\wedge\cdots \wedge i (Y_{n-q})g.
\end{align*}
Then, we can define the Laplacian on a Hessian manifold as follows
$$\Delta =\partial \delta+\delta\partial.$$
\end{definition}
By the above definition, Shima \cite{Shima86} showed that 
$$\Delta g=\beta.$$
Therefore, interestingly, one can obtain that the Hesse flow can be written as 
$$\frac{\partial}{\partial t}g(t)=2\Delta g(t).$$

In the following of this section, we consider a self-similar solution to the Hesse flow:
$$g(t)=\sigma(t)\psi^*(t)g(0),$$
where $g(0)$ is the Hessian metric $g$, $\sigma(t):\mathbb{R}\rightarrow \mathbb{R_+}$ is a smooth function and $\psi(t):M\rightarrow M$ is a 1-parameter family of diffeomorphisms.
By differentiating, we have
\begin{equation}\label{e1}
2\beta(g(t))=\frac{d}{dt}\sigma(t)\psi^*(t)g(0)+\sigma(t)\psi^*(t)(\mathcal{L}_Xg(0)),
\end{equation}
where $\mathcal{L}_X$ denotes the Lie derivative by the time dependent vector field $X$ such that $X(\psi(t)(p))=\frac{d}{dt}(\psi(t)(p))$ for any $p\in M$. 
\begin{claim}\label{cl1}
$\beta(cg)=\beta(g)$ for any positive constant $c\in\mathbb{R}$.
\end{claim}

\begin{proof}
Let $v_g$ and $v_{cg}$ be the volume forms of $g$ and $cg$, respectively. Assume that $\alpha^c$ is the first Koszul form for $cg$. By definition \eqref{1st}, we have
$$D_Xv_{cg}=\alpha^c(X)v_{cg}.$$
From this and the definition of the volume form, we obtain
$$D_Xv_g=\alpha^c(X)v_g.$$
Therefore, we have 
$$\alpha^c=\alpha.$$
From this and the definition of the second Koszul form \eqref{2nd}, 
$$\beta(cg)=D\alpha^c=D\alpha=\beta(g).$$
\end{proof}

By \eqref{e1} and Claim \ref{cl1}, one has
\begin{equation}\label{e2}
2\beta(g(t))=\frac{d}{dt}\sigma(t)g(t)+\mathcal{L}_Yg(t),
\end{equation}
where $Y(t)=\sigma(t)X(t)$.
Therefore, we define self-similar solutions to the Hesse flow as follows:

\begin{definition}
Let $(M,D,g=Dd\varphi)$ be a Hessian manifold. If there exist a vector field $X$ and $\lambda\in\mathbb{R}$, such that
\begin{equation}\label{HS}
\beta-\frac{1}{2}\mathcal{L}_Xg=\lambda g,
\end{equation}
then, $M$ is called a Hesse soliton.
If $\lambda>0,\lambda=0, \lambda<0$, then the Hesse soliton is called expanding, steady or shrinking, respectively. 
If there exists a smooth function $f$ on $M$ such that $X=\grad f$, that is,
\begin{equation}\label{GHS}
\beta-\nabla\nabla f=\lambda g,
\end{equation}
then the Hesse soliton $(M,D,g,f)$ is called a gradient Hesse soliton, where $\nabla\nabla f$ is the Hessian of $f$. $f$ is called a potential function.

\end{definition}
Hesse-Einstein manifolds are trivial solutions of Hesse solitons. Therefore, if a Hesse soliton is Hesse-Einstein, then it is called trivial.
If a Hesse soliton is a proper Hessian manifold, then it is called a proper Hesse soliton.

\quad\\


\section{Existence and non-existence theorems for\\ Hesse solitons}

In this section, we show some existence and non-existence theorems for Hesse solitons.

\begin{theorem}\label{main}~

$(1)$ There exist no compact shrinking Hesse solitons.

$(2)$ Any compact steady Hesse soliton is non proper and trivial.
\end{theorem}

Unlike in the case (1), in the case (2), we remark that there exist non proper and trivial steady Hesse solitons.
We will consider it later.

We first show that the following lemma.
\begin{lemma}\label{key}
On any Hessian manifold, the following formula holds.
\begin{align}\label{e3}
\frac{1}{2}\Delta R
=&~~\nabla_i\nabla_j\alpha_k\gamma_{ijk}-\nabla_r\nabla_r\alpha_i\alpha_i+|\nabla\gamma|^2
+|{\rm Rm}|^2+|\Ric|^2\\
&+R_{ij}\beta_{ij}-R_{ij}\nabla_i\alpha_j\notag.
\end{align}
\end{lemma}
\begin{proof}
Since 
$$\gamma_{ijk}=\frac{1}{2}\frac{\partial g_{ij}}{\partial x_k}~~~\text{and}~~~g_{ij}=\partial_i\partial_j\varphi,$$ we have
$$\nabla_i\gamma_{jkl}=\frac{1}{2}\partial_i\partial_j\partial_k\partial_l\varphi-(\gamma_{rij}\gamma_{rkl}+\gamma_{rik}\gamma_{rjl}+\gamma_{ril}\gamma_{rjk}).$$
This means that $\nabla_i\gamma_{jkl}$ is symmetry with respect to $i,j,k,l$.

A direct computation shows that
\begin{align*}
\nabla_r\nabla_r\gamma_{ijk}
=&~\nabla_r\nabla_i\gamma_{rjk}\\
=&~\nabla_i\nabla_r\gamma_{rjk}
+R_{rirp}\gamma_{pjk}+R_{rijp}\gamma_{rpk}+R_{rikp}\gamma_{rjp}\\
=&~\nabla_i\nabla_r\gamma_{rjk}
+(\gamma_{rpt}\gamma_{tir}-\alpha_t\gamma_{tip})\gamma_{pjk}\\
&+(\gamma_{rpt}\gamma_{tij}-\gamma_{rjt}\gamma_{tip})\gamma_{rpk}
+(\gamma_{rpt}\gamma_{tik}-\gamma_{rkt}\gamma_{tip})\gamma_{rjp}\\
=&~\nabla_i\nabla_j\alpha_k
+\gamma_{rpt}(\gamma_{tir}\gamma_{pjk}+\gamma_{tij}\gamma_{rpk}+\gamma_{tik}\gamma_{rjp})\\
&-\alpha_t\gamma_{tip}\gamma_{pjk}-\gamma_{rjt}\gamma_{tip}\gamma_{rpk}-\gamma_{rkt}\gamma_{tip}\gamma_{rjp},
\end{align*}
where the first and third equalities follow from the symmetric property of $\nabla_i\gamma_{jkl}$  with respect to $i,j,k,l$, and the second one follows from the Ricci identity.
From this, one has
\begin{align}\label{dg}
\frac{1}{2}\Delta|\gamma|^2
=&~\nabla_r\nabla_r\gamma_{ijk}\gamma_{ijk}+|\nabla\gamma|^2\\
=&~\{\nabla_i\nabla_j\alpha_k
+\gamma_{rpt}(\gamma_{tir}\gamma_{pjk}+\gamma_{tij}\gamma_{rpk}+\gamma_{tik}\gamma_{rjp})\notag\\
&-\alpha_t\gamma_{tip}\gamma_{pjk}-\gamma_{rjt}\gamma_{tip}\gamma_{rpk}-\gamma_{rkt}\gamma_{tip}\gamma_{rjp}\}\gamma_{ijk}+|\nabla\gamma|^2.\notag
\end{align}
Substituting 
\begin{align*}
|{\rm Rm}|^2
=&~R_{ijkl}R_{ijkl}\\
=&~\gamma_{rpt}\gamma_{tir}\gamma_{pjk}\gamma_{ijk}+\gamma_{rpt}\gamma_{tij}\gamma_{rpk}\gamma_{ijk}
-\gamma_{rjt}\gamma_{tip}\gamma_{rpk}\gamma_{ijk}-\gamma_{rkt}\gamma_{tip}\gamma_{rjp}\gamma_{ijk},
\end{align*}
into \eqref{dg}, we have
\begin{align*}
\frac{1}{2}\Delta|\gamma|^2
=&~\nabla_i\nabla_j\alpha_k\gamma_{ijk}+|\nabla\gamma|^2
+|{\rm Rm}|^2\\
&+\gamma_{rpk}\gamma_{tik}\gamma_{rjp}\gamma_{ijk}-\alpha_t\gamma_{tip}\gamma_{pjk}\gamma_{ijk}.
\end{align*}
From this, \eqref{ric},
\begin{align*}
|\Ric|^2=\gamma_{skr}\gamma_{rjs}\gamma_{kip}\gamma_{ijp}-\gamma_{skr}\gamma_{rjs}\alpha_{i}\gamma_{ijk}
-\alpha_r\gamma_{rjk}\gamma_{ikp}\gamma_{pji}
+\alpha_r\alpha_i\gamma_{rjk}\gamma_{ijk},
\end{align*}
and
$$\nabla_i\alpha_j=\beta_{ij}-\gamma_{rij}\alpha_r,$$
one has
\begin{align*}
\frac{1}{2}\Delta|\gamma|^2
=&~\nabla_i\nabla_j\alpha_k\gamma_{ijk}+|\nabla\gamma|^2
+|{\rm Rm}|^2+|\Ric|^2
+\gamma_{ijk}\gamma_{rjk}\alpha_p\gamma_{ipr}-\alpha_i\alpha_r\gamma_{ijk}\gamma_{rjk}\\
=&~~\nabla_i\nabla_j\alpha_k\gamma_{ijk}+|\nabla\gamma|^2
+|{\rm Rm}|^2+|\Ric|^2
+R_{ij}\beta_{ij}-R_{ij}\nabla_i\alpha_j.
\end{align*}
Therefore, we have
\begin{align*}
\frac{1}{2}\Delta R
=&~~\nabla_i\nabla_j\alpha_k\gamma_{ijk}-\nabla_r\nabla_r\alpha_i\alpha_i+|\nabla\gamma|^2
+|{\rm Rm}|^2+|\Ric|^2\\
&+R_{ij}\beta_{ij}-R_{ij}\nabla_i\alpha_j\notag.
\end{align*}
\end{proof}

By using the above lemma, one can show Theorem \ref{main}.
\begin{proof}[Proof of Theorem $\ref{main}$]
By taking the trace of \eqref{HS}, we have
$$\beta_{ii}-\dive X=\lambda n.$$
From this, one has
\begin{align*}
\lambda\, n\, \vol(M,g)
=&~\int_M\beta_{ii} -\dive X v_g\\
=&~\int_M\nabla_i\alpha_i+|\alpha|^2-\dive X v_g\\
=&~\int_M|\alpha|^2 v_g\geq0,
\end{align*}
where the last equality follows from Stokes' theorem. 
Hence, one has $\lambda\geq0$.
Therefore, there exist no compact shrinking Hesse solitons.

If $\lambda=0$, then one has $\alpha=0$. Furthermore, by Lemma \ref{key}, we have
\begin{align*}
\frac{1}{2}\Delta R
=&~|\nabla\gamma|^2+|{\rm Rm}|^2+|\Ric|^2.
\end{align*}
By Green's formula, one has
$$\int_M|\Ric|^2+|\Rm|^2+|\nabla\gamma|^2 v_g=0.$$
Therefore, $M$ is flat, in particular $R=0.$
From this and \eqref{scal}, we have $\gamma=0$, that is, $D=\nabla.$ 
Furthermore, we also have $\beta=0$. Therefore, it is also trivial.
\end{proof}

As mentioned above, we consider the properness of steady Hesse solitons.
The equation of steady Hesse solitons is $$\beta-\frac{1}{2}\mathcal{L}_Xg=0.$$
From the proof of (2) of Theorem \ref{main}, any compact steady Hesse soliton is flat and non proper. Hence, any compact steady Hesse soliton is 
$$\mathcal{L}_Xg=0.$$
One can construct many examples of steady Hesse solitons by taking the vector field $X$ as Killing vector field:

\begin{proposition}
Let $(M,g,X)$ be a non proper Hessian manifold with a Killing vector field $X$ $($and the Levi-Civita connection $\nabla)$. Then, $\beta$ and $\mathcal{L}_Xg$ vanishes, and therefore, $(M,g,X)$ is a steady Hesse soliton.
\end{proposition}

We remark that if $\alpha=\alpha'$, then $D=\nabla$ on compact Hessian manifolds. In fact, since the volume form is parallel, $\alpha'(X)=D'_Xv_g=(2\nabla-D)v_g=-D_Xv_g=-\alpha(X)$, that is, $\alpha=-\alpha'$, thus one has $\alpha=\alpha'=0.$ By Lemma \ref{key} and Green's formula, one has 
$$\int_M|{\rm Rm}|^2+|\Ric|^2+|\nabla\gamma|^2 v_g=0.$$
From this and \eqref{scal}, we have $\gamma=0$, that is, $D=\nabla.$
On a complete Hessian manifold, it is well known that E. Calabi obtained the same conclusion (cf. \cite{Calabi58}), that is, any complete Hessian manifold with $\alpha=0$ satisfies $D=\nabla.$

From the above arguments, we consider a more general problem. Obviously, if a Hessian manifold $M$ is non proper, then $\beta=\beta'(=0)$. 
In particular, $\nabla\alpha=0.$
In fact, since $\alpha=-\alpha'$, one has $\beta'=D'\alpha'=-D'\alpha=-(2\nabla-D)\alpha=\beta-2\nabla\alpha.$ Therefore,
$\beta'-\beta=2\nabla\alpha.$

However, the converse is not true, that is, even if $\beta=\beta'$, $M$ might not satisfies that $D=\nabla$. In fact, the following example satisfies $\beta=\beta'=\frac{n}{2}g$, but $D\not=\nabla.$ This means that there exist proper Hesse-Einstein manifolds.

\begin{example}\label{dualHE}
Let 
$$\Omega=\left\{x\in\mathbb{R}^n;x_n>\sqrt{\displaystyle\sum_{i=1}^{n-1}(x_i)^2}\right\}~~~\text{and}~~~\varphi=-\log\left(x_n^2-\left(\displaystyle\sum_{i=1}^{n-1}(x_i)^2\right)\right).$$ Then, $(\Omega,D,g=Dd\varphi)$ is a Hessian structure on $\Omega$.
\end{example}

From the above argument, it is interesting to consider the problem: ``Does there exist non trivial Hesse soliton with $\beta=\beta'$ (that is, $\nabla\alpha=0)$?"

We first consider a complete Einstein Hessian manifold with a non-negative Einstein constant $\lambda$, that is,
a Hessian manifold with $\Ric=\lambda g$ with $\lambda\geq0.$

\begin{proposition}
Any complete Einstein Hessian manifold with a non-negative Einstein constant $\lambda$ and $\beta=\beta'$ is flat and $\nabla\gamma=0.$
\end{proposition}

\begin{proof}
By the assumption,
$$0=\nabla_i\alpha_j=\beta_{ij}-{\gamma^r}_{ij}\alpha_r=\beta_{ij}-g^{rs}\gamma_{rij}\alpha_s.$$
From this, \eqref{e3} and the assumption, 
\begin{align*}
\frac{1}{2}\Delta R
=&~|\nabla\gamma|^2+|{\rm Rm}|^2+|\Ric|^2+\lambda \beta_{ii}\\
=&~|\nabla\gamma|^2+|{\rm Rm}|^2+|\Ric|^2+\lambda |\alpha|^2\\
\geq&  \frac{1}{n}R^2,
\end{align*}
where the last inequality follows from the Schwarz inequality.
Since the Ricci curvature is non-negative, by the Omori-Yau maximum principle (cf. \cite{Omori}, \cite{Yau}), $R=0.$ Hence $M$ is flat and $\nabla\gamma=0.$
\end{proof}

\begin{lemma}\label{ghsbbkey}
Any Hesse soliton with $\beta=\beta'$ satisfies that
$\dive X$ is constant.
\end{lemma}

\begin{proof}
Since 
$$0=\nabla_i\alpha_j=\beta_{ij}-{\gamma^r}_{ij}\alpha_r=\beta_{ij}-g^{rs}\gamma_{rij}\alpha_s,$$
we have
$$\beta_{ii}=|\alpha|^2.$$
Hence, one has
$$\nabla_k\beta_{ii}=0.$$
By the equation of Hesse solitons \eqref{HS}, 
$$0=\nabla_k\beta_{ii}=\nabla_k(\dive X+n\lambda)=\nabla_k\dive X,$$
which implies that 
$\dive X$ is constant.
\end{proof}

By Lemma \ref{ghsbbkey}, one can show the following:

\begin{proposition}\label{cptghsbb0}
If compact Hesse solitons satisfy $\beta=\beta'$, then $\dive X=0$.
\end{proposition}

\begin{proof}
By Lemma \ref{ghsbbkey}, $\dive X$ is constant, say $C$. By Stokes' theorem,
$$0=\int_M\dive X v_g=C \, \vol(M).$$
Thus, $C=0$, that is, $\dive X=0$.
\end{proof}

In particular, if $M$ is gradient, by the standard maximum principle, one can obtain the following.

\begin{corollary}\label{cptghsbb}
Any compact gradient Hesse soliton with $\beta=\beta'$ is trivial.
\end{corollary}

A similar result for complete Hesse solitons can be obtained. 

\begin{proposition}\label{complghsbb}
Any complete gradient Hesse soliton with $\beta=\beta'$ and non-negative Ricci curvature is trivial.
\end{proposition}

\begin{proof}
By Lemma \ref{ghsbbkey}, 
\begin{align*}
\Delta |\nabla f|^2
=&~ 2|\nabla\nabla f|^2+2\Ric (\nabla f,\nabla f)+2g(\nabla f,\nabla \Delta f)\\
=&~ 2|\nabla\nabla f|^2+2\Ric (\nabla f,\nabla f)\geq0.\notag
\end{align*}
Hence, Omori-Yau maximum principle shows that $|\nabla f|^2$ is constant, say $C$.
Assume that $C>0$.
Since
$\nabla\nabla f=0$, $\Delta f=0.$
From this, 
\begin{align*}
\Delta e^f
=~ |\nabla f|^2e^f+\Delta fe^f
=~ |\nabla f|^2e^f>0.\notag
\end{align*}
By Omori-Yau maximum principle again, $e^f$ is constant, that is, $f$ is constant, which is a contradiction.

\end{proof}

By Corollary \ref{cptghsbb}, it is interesting to consider non trivial gradient Hesse solitons from the point of view of information geometry.

\begin{corollary}
Any compact non trivial gradient Hesse soliton is proper.
\end{corollary}

\begin{proof}
By Corollary \ref{cptghsbb}, $\nabla\alpha\not=0$ at some point $p\in M$, i.e., on some open set $\Omega\ni p$. Hence, $\nabla\gamma\not=0$ on $\Omega$. In fact, if $\nabla\gamma=0$ at $q\in \Omega$, then we have $\nabla\alpha=0$ at $q$, which is a contradiction.

Therefore, $\gamma\not=0$ on some set $\tilde\Omega$ of $M$, which means that the soliton is proper.

\end{proof}

By the same argument, one can show the following.

\begin{corollary}
Any complete non trivial gradient Hesse soliton with non-negative Ricci curvature is proper.
\end{corollary}

\quad\\

\section{Dual Hessian structure}\label{dual}

In this section, we consider the dual space of a Hessian manifold $(M,D,g)$ and show that one can understand the dual space of a Hesse-Einstein manifold as a Hesse soliton.

\begin{theorem}\label{dualmain}
Let $(M,D,g)$ be a Hesse soliton, 
$$\beta-\frac{1}{2}\mathcal{L}_Xg=\lambda g,$$
then the dual space $(M,D',g)$ is also a Hesse soliton which satisfies that 
$$\beta'-\frac{1}{2}\mathcal{L}_{(X-2\alpha^\sharp)}g=\lambda g,$$
where $\sharp$ is 
a musical isomorphism $\sharp: TM^*\rightarrow TM$,
$$g(\alpha^\sharp ,X)=\alpha(X),$$
for any vector field $X$ on $M$.

\end{theorem}

\begin{proof}
By the definition of the musical isomorphism $\sharp$,
\begin{align*}
(\nabla \alpha)(Y,Z)
=&~(\nabla_Y\alpha)(Z)\\
=&~Y\alpha(Z)-\alpha(\nabla _YZ)\\
=&~Yg(\alpha^\sharp, Z)-g(\alpha^\sharp,\nabla _YZ)\\
=&~g(\nabla_Y\alpha^\sharp,Z)+g(\alpha^\sharp,\nabla_YZ)-g(\alpha^\sharp,\nabla_YZ)\\
=&~g(\nabla_Y\alpha^\sharp,Z).
\end{align*}
Since $\beta$ and $\beta'$ are symmetric 2 forms, $\nabla\alpha=\frac{1}{2}(\beta'-\beta)$ is also a symmetric 2 form. Thus, 
\begin{align*}
2(\nabla\alpha)(Y,Z)
=&~(\nabla\alpha)(Y,Z)+(\nabla \alpha)(Z,Y)\\
=&~g(\nabla_Y\alpha^\sharp,Z)+g(\nabla_Z\alpha^\sharp,Y)\\
=&~\mathcal{L}_{\alpha^\sharp} g(Y,Z).
\end{align*}
Since $2\nabla\alpha=\beta-\beta'$ and $(M,D,g)$ is a Hesse soliton
$$\beta-\frac{1}{2}\mathcal{L}_X(Y,Z)=\lambda g,$$
one has
\begin{align*}
\beta'(Y,Z)
=&~\beta(Y,Z)-2(\nabla\alpha)(Y,Z)\\
=&~\frac{1}{2}\mathcal{L}_Xg(Y,Z)+\lambda g(Y,Z)-\mathcal{L}_{\alpha^\sharp}g(Y,Z)\\
=&~\frac{1}{2}\mathcal{L}_{(X-2\alpha^\sharp)}g(Y,Z)+\lambda g(Y,Z).
\end{align*}
\end{proof}

We consider gradient Hesse solitons. If the first Koszul form $\alpha$ is exact, that is, $\alpha=dF$ for some smooth function $F$ on $M$, then the Hesse soliton of the dual space is also gradient.
\begin{corollary}
Let $(M,D,g,f)$ be a gradient Hesse soliton, 
such that the first Koszul form is exact, that is, $\alpha=dF$ for some smooth function $F$ on $M$. Then the dual space $(M,D',g)$ is also a gradient Hesse soliton with the potential function $f-2F.$
\end{corollary}
\begin{proof}
Since $\alpha=dF$, we have 
$$g(\alpha^\sharp,Y)=\alpha(Y)=dF(Y)=XF=g(\nabla F,Y),$$
for any vector field $Y$ on $M$.
Thus, we have
$$\alpha^\sharp=\nabla F.$$
By Theorem \ref{dualmain}, the proof is complete.
\end{proof}

One can understand the dual space of Hesse-Einstein manifolds as Hesse solitons.
\begin{corollary}
Let $(M,D,g)$ be a Hesse-Einstein manifold, 
$$\beta=\lambda g,$$
then the dual space is a Hesse soliton $(M,D',g,-2\alpha^\sharp)$, that is, it satisfies 
$$\beta'-\frac{1}{2}\mathcal{L}_{(-2\alpha^\sharp)}g=\lambda g.$$
\end{corollary}

\quad\\

\bibliographystyle{amsbook}

\end{document}